\documentclass{amsart}
\usepackage{amsmath,amsthm,amssymb,amsfonts,graphicx,color}

\newtheorem*{thm1}{Theorem 4.12}
\newtheorem*{thm2}{Theorem 4.16}

\newtheorem{theorem}{Theorem}[section]
\newtheorem{lemma}[theorem]{Lemma}

\newtheorem{proposition}[theorem]{Proposition}
\newtheorem{definition}[theorem]{Definition}

\newtheorem{corollary}[theorem]{Corollary}

\theoremstyle{remark}
\newtheorem{remark}[theorem]{Remark}
\newtheorem{example}[theorem]{Example}
\newtheorem*{ack}{Acknowledgements}

\begin{document}

\title{Betti diagrams from graphs}

\author[A. Engstr\"om]{Alexander Engstr\"om}
\address{Department of Mathematics, Aalto University, Helsinki, Finland}
\email{alexander.engstrom@aalto.fi}

\author[M. T. Stamps]{Matthew T. Stamps}
\address{Department of Mathematics, Aalto University, Helsinki, Finland}
\email{matthew.stamps@aalto.fi}
\date{\today}

\begin{abstract}
The emergence of Boij-S\"oderberg theory has given rise to new connections between combinatorics and commutative algebra.  Herzog, Sharifan, and Varbaro recently showed that every Betti diagram of an ideal
with a $k$-linear minimal resolution arises from that of the Stanley-Reisner ideal of a simplicial complex.  In this paper, we extend their result for the special case of $2$-linear resolutions using purely combinatorial methods.  Specifically, we show bijective correspondences between Betti diagrams of ideals with 2-linear resolutions, threshold graphs, and anti-lecture hall compositions.  Moreover, we prove that any Betti diagram of a module with a $2$-linear resolution is realized by a direct sum of Stanley-Reisner rings associated to threshold graphs.  Our key observation is that these objects are the lattice points in a normal reflexive lattice polytope.
\end{abstract}



\maketitle

\section{Introduction}

A fundamental problem in commutative algebra is to characterize the coarsely graded Betti numbers of the finitely generated graded modules over a fixed polynomial ring.  Originating with Hilbert in the 1890's, this task largely eluded mathematicians until 2006, when Boij and S\"oderberg introduced the following relaxation:  Instead of trying to determine whether or not a table of nonnegative integers is the Betti diagram of a module, one should try to determine if some rational scalar of the table is the Betti diagram of a module.  This shifted the viewpoint to studying rays in a rational cone and with this new geometric picture, the subject has seen a great deal of progress over the last six years.  In particular, the idea led Boij and S\"oderberg to conjecture that every Betti diagram of a module can be decomposed in a specific and predictable way \cite{BS1}.  Eisenbud and Schreyer proved this for Cohen-Macaulay modules \cite{ES} and Boij and S\"oderberg later extended that proof to the general setting \cite{BS2}.

A natural question that arises from Boij-S\"oderberg theory is the following:  If a module is constructed from a combinatorial object, e.g., the edge ideal of a graph or the Stanley-Reisner ideal of a simplicial complex, can any of the combinatorial properties of that object be seen in the Boij-S\"oderberg decomposition of the module?  Herzog, Sharifan, and Varbaro recently gave an elegant partial answer to this question \cite{HSV} for the special case of ideals with $k$-linear resolutions by showing that every Betti diagram of an ideal with a $k$-linear minimal resolution can be realized by the Stanley-Reisner ideal of a certain simplicial complex.  More specifically, they prove that from the coefficients of a Boij-S\"oderberg decomposition of a $k$-linear Betti diagram, one obtains an $O$-sequence which, by a famous result of Eagon and Reiner along with Macaulay's theorem, yields a simplicial complex with the desired properties.  Nagel and Sturgeon employ a similar approach to show that the $k$-linear Betti diagrams can be realized with hyperedge ideals of $k$-uniform Ferrers hypergraphs \cite{NS}. 

In this paper, we restrict our attention to the case of $2$-linear resolutions and give an alternate characterization of the Betti diagrams of ideals with $2$-linear minimal resolutions using purely combinatorial means.  We show that every Betti diagram from an ideal with a $2$-linear resolution is realized by a Stanley-Reisner ring constructed from a threshold graph and that this correspondence is a bijection. 

\begin{thm1}
For every $2$-linear ideal $I$ in $S$, there is a unique threshold graph $T$ on $n+1$ vertices with $\beta(S/I) = \beta(\mathbf{k}[T])$.
\end{thm1}

Moreover, for any such ideal, we give an efficient algorithm for constructing its corresponding threshold graph that avoids expensive computations like Hochster's formula; rather, we can generate all such Betti diagrams recursively with affine transformations, avoiding operators such as Ext and Tor.  Even more interesting, we find that these diagrams are the lattice points of a normal reflexive lattice simplex that is combinatorially equivalent to a simplex of anti-lecture hall compositions and, from this geometric picture, we prove that any Betti diagram of a \emph{module} with a $2$-linear resolution arises from a direct sum of Stanley-Reisner rings constructed from threshold graphs.

\begin{thm2}
For every graded $S$-module $M$ with $2$-linear minimal free resolution and $\beta_{0,0}(M) = m$, there exists a collection of $m$ threshold graphs $\{T_1,\dots,T_m\}$, not necessarily distinct, such that $\beta(M) =  \beta(\mathbf{k}[T_1] \oplus \cdots \oplus \mathbf{k}[T_m])$. 
\end{thm2}

The paper is organized as follows:  In Section 2, we give a quick review of the necessary concepts from commutative algebra and Boij-S\"oderberg theory.  In Section 3, we interpret the main theorem of Boij-S\"oderberg theory in terms of linear algebra for the special case of modules with $k$-linear minimal resolutions.  We prove our main theorems in Section 4 and conclude with some interesting connections to discrete geometry in Section 5.

\section{Preliminaries}

We begin with a review of the basic definitions and theorems from Boij-S\"oderberg theory.  For a more detailed introduction to this topic, we recommend \cite{floystad}. 

\subsection{Commutative Algebra}

Let $\mathbf{k}$ be a field and $S = \mathbf{k}[x_1,\dots,x_n]$. 
 For any finitely generated graded $S$-module $M$, let $M_i$ 
 denote its graded piece of degree $i$ and let $M(d)$ denote the  
 \emph{twisting} of $M$ by $d$, i.e. the module where 
 $M(d)_i \cong M_{i+d}$.  
A \emph{minimal graded free resolution} of $M$ is an exact 
complex 
$$ 0 \leftarrow M \leftarrow F_0 \leftarrow  F_1 \leftarrow  \cdots \leftarrow F_l$$
where each $F_i$ is a graded free $S$-module of the form
 $\oplus_{j \in \mathbb{Z}} S(-j)^{\beta_{i,j}}$ such that the number 
 of basis elements is minimal and each map is graded.  

The value $\beta_{i,j}$ is called the $i$th \emph{graded Betti number} of 
degree $j$.  These numbers are a refinement of the ordinary Betti numbers $\beta_i = \sum_j \beta_{i,j}$ and are independent of the choice of resolution of $M$, thus yielding an important numerical invariant of $M$.  We often express the graded Betti numbers in a two-dimensional array called the \emph{Betti diagram} of $M$, denoted by $\beta(M)$.  Since $\beta_{i,j} = 0$ 
whenever $i > j$, it is customary to write $\beta(M)$ such that 
$\beta_{i,j}$ is in position $(j-i,i)$.  That is, 
\[
\beta(M) =
\left[
\begin{array}{cccc} 
\beta_{0,0} & \beta_{1,1} &  \cdots & \beta_{l,l} \\ 
\beta_{0,1} & \beta_{1,2} &  \cdots & \beta_{l,l+1} \\ 
\vdots & \vdots & \ddots & \vdots \\ 
\beta_{0,r} & \beta_{1,r+1} &  \cdots & \beta_{l,l+r}  
\end{array}
\right].
\]
A Betti diagram is called \emph{pure} if every 
column has at most one nonzero entry, i.e. for each 
$i \in \{0,\dots,l\}$, $\beta_{i,j} \neq 0$ for at most one $j \in \mathbb{Z}$.  

\subsection{Boij-S\"oderberg theory}

Let $\mathbb{Z}_{deg}^{n+1}$ denote the set of strictly increasing nonnegative 
integer sequences $\mathbf{d} = (d_0,\dots,d_s)$ with $s \leq n$, called 
\emph{degree sequences}, along with the partial 
order given by $$(d_0,\dots,d_s) \geq (e_0,\dots,e_t)$$ whenever 
$s \leq t$ and $d_i \geq e_i$ for all $i \in \{0,\dots,s\}$.  To every $\mathbf{d} \in 
\mathbb{Z}_{deg}^{n+1}$, we associate a pure Betti diagram $\pi(\mathbf{d})$ 
with entries defined as follows:
\[
\pi_{i,j}(\mathbf{d})=\left\{
\begin{array}{ll}
\displaystyle \prod_{k\neq 0,i} 
\left|
\frac{d_k-d_0}{d_k-d_i}
\right| 
& i \geq 0, j=d_i, \\
0 & \textrm{otherwise.}
\end{array}
\right.
\]
The main theorem of Boij-S\"oderberg theory states that the Betti diagram of any 
graded $S$-module can be written as a positive rational combination of $\pi
(\mathbf{d})$'s.  It was originally conjectured by Boij and S\"oderberg \cite{BS1}, proven for Cohen-Macaulay modules Eisenbud and Schreyer \cite{ES}, and generalized to the form below by Boij and S\"oderberg \cite{BS2}.

\begin{theorem}[Boij-S\"oderberg] \label{thm:boijSoderberg}
For every graded $S$-module, $M$, there exists a vector $c \in \mathbb{Q}_{\geq 0}^p$ 
and a chain of degree sequences $\mathbf{d}^1 < \mathbf{d}^2 < \cdots < \mathbf{d}^p$ in 
$\mathbb{Z}_{deg}^{n+1}$ such that 
$$\beta(M) = c_1 \pi(\mathbf{d}^1) + \cdots + c_p \pi (\mathbf{d}^p).$$
\end{theorem}

The combination in Theorem \ref{thm:boijSoderberg} is called a \emph{Boij-S\"oderberg decomposition} of $M$ and the entries of $c$ are called \emph{Boij-S\"oderberg coefficients}.  This decomposition is not unique in general, but there is a simple algorithm for computing a set of coefficients that satisfy the theorem, see \cite{floystad}. 

\section{Betti diagrams of 2-linear resolutions}

An ideal $I$ in $S$ is called \emph{$k$-linear} if $\beta_{i,j}(I) = 0$ whenever $j-i \neq k-1$.  If $I$ is $2$-linear, then the Betti diagram of $M = S/I$ looks like \[ 
 \beta(M) =
 \left[
 \begin{array}{cccccc}
1 & \cdot & \cdot & \cdot & \cdots & \cdot \\
\cdot & \beta_1 &  \beta_2 & \beta_3 &  \cdots & \beta_s \\
\end{array}
\right]
\]
for some $s \leq n$.  Our aim is to translate the statement of Theorem \ref{thm:boijSoderberg}, for $S$-modules with $2$-linear resolutions, into linear algebraic terms.  For this, it will be convenient to consider the \emph{reduced Betti vector} $\omega(M) = [\beta_1,\dots,\beta_s]$ in place of $\beta(M)$.  

If $M$ is a $2$-linear $S$-module, then every $\mathbf{d}^l$ in Theorem \ref{thm:boijSoderberg} is of the form $(0,2,\dots,l+1)$.  So, let $\pi^l = \pi(\mathbf{d}^l)$, $\omega^l$ be the reduced Betti vector corresponding to $\pi^l$, and $\Omega$ be the lower-diagonal $n \times n$ matrix whose $l^{\rm th}$ row is $\omega^l$.   We leave it to the reader to verify the following:

\begin{lemma}
The matrix $\Omega$ is invertible and has $ij$--entry $\omega_{j}^i = j{i+1 \choose j+1}$.  Moreover, the $ij$--entry of $\Omega^{-1}$ is $(-1)^{i-j}\frac{1}{i}{i+1 \choose j+1}.$
\end{lemma}

Since any subset of row vectors in $\Omega$ forms a chain in $\mathbb{Z}_{deg}^{n+1}$, we can replace the vector $c \in \mathbb{Q}_{> 0}^p$ in Theorem \ref{thm:boijSoderberg} with a vector $c \in \mathbb{Q}_{\geq 0}^{n}$ such that $\sum_i c_i = \beta_{0,0}(M)$.   

\begin{theorem}\label{thm:BS} For every $2$-linear (graded) $S$-module $M$ with $\beta_{0,0}(M) = m$, $$\beta(M) = c_1 \pi^1 + \cdots + c_n \pi^n$$ where $c = \omega(M)\Omega^{-1} \in \mathbb{Q}_{\geq 0}^n$ and $\sum\limits_{i} c_i = m$. \end{theorem}

\begin{remark}\label{rem:lattice} When $\beta_{0,0}(M) = 1$, Theorem \ref{thm:BS} asserts that $\omega(M)$ is a lattice point in the $(n-1)$-dimensional simplex spanned by row vectors of $\Omega$. \end{remark}

We conclude this section with some classic examples of $2$-linear ideals that arise from graph theory.  A \emph{graph} $G$ consists of a finite set $V(G)$, called the \emph{vertex set}, and a subset $E(G)$ of $\binom{V(G)}{2}$, called the \emph{edge set}.   To simplify notation, we write $uv$ instead of $\{u,v\}$ for each edge in $G$.  For any subset of vertices $W \subset V(G)$, the \emph{induced subgraph} $G[W]$ is the graph with vertex set $W$ and edge set $E(G) \cap \binom{W}{2}$.  If $W = V(G) \setminus S$ for some $S \subseteq V(G)$, we write $G \setminus S$ instead of $G[W]$.  A subgraph $C$ of the form $V(C) = \{v_1, \dots, v_l\}$ and $E(C) = \{ v_iv_{i+1} \ | \ 1 \leq i < l \} \cup \{v_1v_l\}$ is called a \emph{cycle} of length $l$.  We say $G$ is  \emph{chordal} if it has no induced cycles of length greater than three or, equivalently, if $E(C) \subsetneq E(G[C])$ for every cycle of length greater than three.  The elements of $E(G[C]) \setminus E(C)$ are called \emph{chords}.  Chordal graphs have many interesting properties that are actively studied in graph theory.  For a thorough introduction to graph theory, we recommend Diestel \cite{diestel2012}.

Given a graph $G$ with vertex set $[n+1] = \{1,\dots,n+1\}$, where $n$ is the number of indeterminants in $S$, let $R = \mathbf{k}[x_1,\dots,x_{n+1}]$,  let $I^c(G) = \langle x_ix_j \mid ij \notin E(G) \rangle \subseteq R$ be the ideal generated by the monomials corresponding to nonedges in $G$, and let $\mathbf{k}[G]$ be the quotient $R / I^c(G)$.  The knowledgeable reader may observe that $I^c(G)$ is the edge ideal of the complement of $G$ and $\mathbf{k}[G]$ is the Stanley-Reisner ring of the clique complex of $G$.  The following theorem was first proved by Fr\"oberg \cite{froberg90} and then by Dochtermann and Engstr\"om \cite{DE}, using topological combinatorics.

\begin{theorem} \label{thm:froberg}
A graph $G$ is chordal if and only if $I^c(G)$ is $2$-linear.  Whenever this is the case, $$\beta_{i,j}(\mathbf{k}[G]) =
\sum_{W \in {V(G) \choose j}} (-1 + \# 
\textrm{components of $G[W]$})$$ for $i=j-1 \geq 1$.
\end{theorem}

\begin{example}\label{ex:iso}
If $G$ consists of $n+1$ isolated vertices, then the ${n+1 \choose i+1}$ induced subgraphs of $G$ with $i+1$ 
vertices each have $i+1$ connected components. Thus, $\beta_{i,i+1}(\mathbf{k}[G])=i{n+1 \choose i+1}$ for each 
$i \geq 1$.
\end{example}

\begin{example}\label{ex:eta}
If $G$ consists of a complete graph on $n$ vertices plus an isolated vertex, $v$, then the $\binom{n}{i}$ 
induced subgraphs of $G$ with $i+1$ vertices that contain $v$ each have two connected components and the remaining induced subgraphs of $G$ (with $i+1$ vertices) are connected.  Thus, $\beta_{i,i+1}(\mathbf{k}[G])={n \choose i}$ for each $i \geq 1$. 
\end{example}

\begin{remark}
If we apply Theorems \ref{thm:BS} and \ref{thm:froberg} to $\mathbf{k}[G]$ for some chordal graph $G$, we get a formula that takes the number of connected components of induced subgraphs of $G$ as input and yields a vector $c \in \mathbb{Q}_{\geq 0}^n$, namely $\omega((\mathbf{k}[G])\Omega^{-1}$, whose entries sum to 1.  It is natural to ask what this formula says if $G$ is not chordal? If the entries of $c$ fail to be nonnegative or sum to 1, then we get a certificate that $G$ is not chordal. Since measuring how far a graph is from being chordal is nontrivial from the viewpoint of complexity, one is inclined to ask if this procedure characterizes chordal graphs.  

Alas, this turns out to not be the case -- there are nonchordal graphs that yield admissible $c$'s -- but these  \emph{false chordal graphs} seem to be few.  Examples of false chordal graphs on six and seven vertices are illustrated in Figure \ref{fake-chordal}.  All other false chordal graphs on seven vertices arise from expanding a (possibly empty) clique of the six vertex graph or coning over all of the six vertex graph.  Computer generated statistics on the size of each class of graphs for a given number of vertices are provided in Table~\ref{fake-table}. 

\end{remark}

\begin{figure} [ht]
\begin{center}\includegraphics[width=0.7\textwidth]{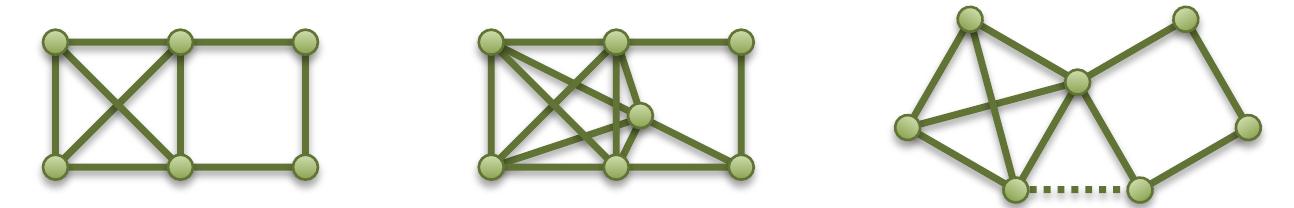} \end{center}
\caption{The single false chordal graph on six vertices along with two examples on seven vertices.  }\label{fake-chordal}
 \end{figure}

\begin{table}[ht]
\centering
\begin{tabular}{r|ccccccc}
 & 1 & 2 & 3 & 4 & 5 & 6 & 7 \\
\hline
Chordal & 1 & 2 & 4 & 10 & 27 & 94 & 393  \\
False chordal & 0 & 0 & 0 & 0 & 0 & 1 & 15  \\ 
Not chordal & 0 & 0 & 0 & 1 & 7 & 62 & 651  \\
\end{tabular}
\caption{Classes of graphs on different numbers of vertices.}
\label{fake-table}
\end{table}

\section{Betti diagrams from graphs}

In this section, we study the Betti diagrams corresponding to a special class of chordal graphs called threshold graphs.  
We show that threshold graphs on a fixed vertex set have distinct Betti diagrams, that every Betti diagram of a chordal graph is that of a threshold graph on the same number of vertices, that every Betti diagram of an $S$-algebra with a $2$-linear resolution is that of a threshold graph on $n+1$ vertices, and that every Betti diagram of an $S$-module with a $2$-linear resolution is that of a direct sum of Stanley-Reisner rings constructed from threshold graphs on $n+1$ vertices, where $n$ is the number of in determinants in $S$.

\subsection{Betti diagrams from threshold graphs}

In a graph $G$, two vertices are said to be \emph{adjacent} if they are contained in an edge of $G$.  A vertex adjacent to no others is called \emph{isolated} and a vertex adjacent to all others is called \emph{dominating}.  For every graph $G$ on $n$ vertices, let $G_\ast$ be the graph on $n+1$ vertices obtained by adding an isolated vertex to $G$ and, similarly, let $G^\ast$ be the graph obtained by adding a dominating vertex to $G$.  A graph $G$ is called \emph{threshold} if it can be constructed from a single vertex and a sequence of the operations $-^\ast$ and $-_\ast$.  It is well-known that if $G$ is chordal, then so are $G_{\ast}$ and $G^{\ast}$, and thus, all threshold graphs are chordal.  We refer to Mahadev and Peled  \cite{mahadevPeled1995} for a survey that includes the following lemma:

\begin{lemma}\label{lem:thresholdCount}
There are $2^{n}$ threshold graphs on
$n+1$ vertices.  Moreover, every threshold graph is determined by a unique sequence of $-^\ast$ and $-_\ast$. 
\end{lemma}

The Betti diagram of a threshold graph can be constructed recursively
in a similar manner to the graph itself.  As such, we can quickly calculate the Betti diagram of a 
threshold graph without the computations in Theorem \ref{thm:froberg}.  

\begin{proposition}\label{prop:threshold}
If $G$ is a chordal graph on $n$ vertices, then $$(1) \ \ 
\omega(\mathbf{k}[G^*]) = [ \ \omega(\mathbf{k}[G]) \ | \ 0 \ ] \text{ and } \ (2) \ \ \omega(\mathbf{k}[G_*]) = \omega(\mathbf{k}[G])
\Lambda + \eta_{n}$$ where $\Lambda$ is 
the $(n-1) \times n$-matrix whose $(i,j)$ position is $1$ if $i = j$ or $j-1$ and 
$0$ otherwise and $\eta_n$ is the vector whose $i$th entry is $\binom{n}{i}$.
\end{proposition}

\begin{proof}
This is a simple application of Theorem \ref{thm:froberg}.  For the first part, any subset of vertices 
containing the dominating vertex in $G^*$ spans a connected graph and 
therefore, the only nonzero parts of $\omega(\mathbf{k}[G^*])$ come from $\omega(\mathbf{k}[G])$.  For the second part, 
we consider 
whether or not a subset of vertices in $G_*$ contains the isolated vertex $v$:  The induced 
subgraphs that \emph{do not} contain $v$ contribute $[ \ \omega(\mathbf{k}[G]) \ | \ 0 \ ]$ to 
$\omega(\mathbf{k}[G_*])$ while those that \emph{do} contain $v$ contribute $[ \ 0 \ | \ \omega(\mathbf{k}[G]) \ ] + 
\eta_{n}$.  \end{proof}

As a corollary, we find that distinct threshold graphs on a fixed number of vertices have distinct Betti diagrams.

\begin{corollary}\label{cor:threshold}
If $T$ and $T'$ are threshold graphs on the same number of vertices and $\omega(\mathbf{k}[T]) = 
\omega(\mathbf{k}[T'])$, then $T \cong T'$.  
\end{corollary}

\begin{proof}
For any chordal graph $G$ on $k$ vertices, $\omega_{k+1}(\mathbf{k}[G_\ast]) \neq \omega_{k+1}(\mathbf{k}[G^\ast]) = 0$ by Proposition \ref{prop:threshold}.  Therefore, since distinct threshold graphs have distinct sequences of $-_\ast$ and $-^\ast$ (Lemma \ref{lem:thresholdCount}), they must also have distinct Betti diagrams.  
\end{proof}

\subsection{Betti diagrams from chordal graphs}

Next, we show that every Betti diagram from a chordal graph arises as 
the Betti diagram of a threshold graph on the same number of vertices.  Moreover, for a given chordal graph, we present an efficient algorithm for constructing its ``threshold representative''.

Let $\sim_\beta$ be the equivalence relation for graphs on $[n+1]$ defined 
by $$G \sim_{\beta} H \text{ if and only if } \beta(\mathbf{k}[G]) = \beta(\mathbf{k}[H])$$ 
and let $[G]_\beta$ denote the equivalence class 
of $G$ with respect to $\sim_\beta$.  For a chordal graph $G$ on $n+1$ vertices, a threshold graph $T$ (on $n+1$ vertices) is called a \emph{threshold representative}  of $G$ if $T \in [G]_{\beta}$.  The next theorem follows from the notion of \emph{algebraic shifting} and can be pieced together from results in \cite{GY,klivans07,woodroofe11}, but we offer a purely graph-theoretic proof instead.

\begin{theorem}\label{thm:threshold}
Every chordal graph $G$ has a unique threshold representative $T$.
\end{theorem}

We proceed with some new machinery:  For a graph $G$ with $v,w \in V(G)$, we define a new graph $G_{v \to w}$ on $V(G)$ with $$E(G_{v \to w}) := (E(G) \setminus \{uv \ | \ u \in N(v;w)\}) \cup \{uw \ | \ u \in N(v;w)\}$$ where $N(x) = \{ y \in V(G) \ | \ xy \in E(G)\}$ is the \emph{neighborhood} of a vertex $x$ and $N(v;w) = N(v) \setminus (\{w\} \cup N(w))$.

\begin{figure} [ht]
\begin{center}\includegraphics[width=0.55\textwidth]{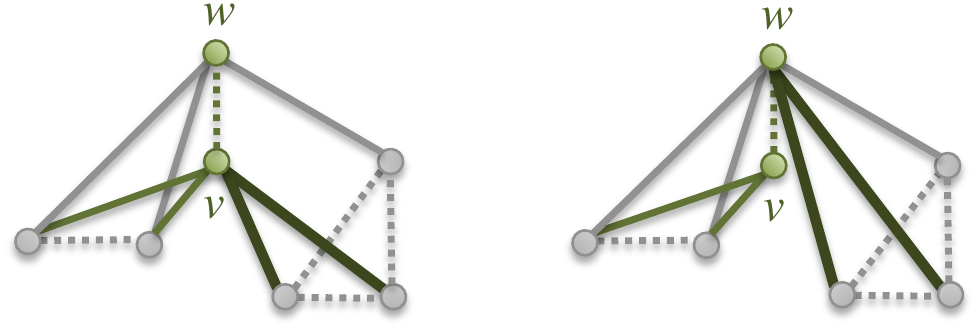} \end{center}
\caption{A comparison of a graph $G$ (left) with $G_{v \to w}$ (right).
}\label{operator}
 \end{figure}

\begin{lemma}\label{lem:chordal}  Let $G$ be a chordal graph. Then $(1)$ If $G$ is a connected with $vw \in E(G)$, then $G_{v \to w}$ is chordal; and $(2)$ If $G$ is disconnected with $v,w \in V(G)$ in separate components, then $G_{v \to w}$ is chordal.
\end{lemma}

\begin{proof}  For each part, we suppose $C$ is a cycle with length $l \geq 4$ in $G' = G_{v \to w}$ and show that $C$ has a chord in $G'$.

In $(1)$, if $w \notin V(C)$, then $C \subseteq G$ since the only new edges of $G'$ contain $w$ and therefore $C$ has at least one chord in $G$.  If every chord of $C$ in $G$ is removed in $G'$, then they must each contain $v$ and thus $G[V(C \setminus v) \cup w]$ is an induced cycle, which is a contradiction.  If $w \in V(C)$, $v \notin V(C)$, and $C$ does not have a chord in $G'$, then $G[V(C) \cup v]$ is an induced cycle since $N(v) \subseteq N(w)$ in $G'$, another contradiction.  If $v,w \in V(C)$, then $vw \in E(C)$ and $xw$ is a chord of $C$ in $G'$, where $x$ is the other neighbor of $v$ in $C$, since $N(v) \subseteq N(w)$ in $G'$.  

In $(2)$, if $w \notin V(C)$, then $C$ contains a chord in $G \setminus w = G' \setminus w \subseteq G'$.  So suppose $w \in V(C)$ and $C$ has no chord in $G'$.  Then $G[V(C \setminus w)]$ is contained in the connected component of either $v$ or $w$ in $G$.  If the former is true, then $G[V(C \setminus w) \cup v]$ is an induced cycle and if the latter is true, then $C$ itself is an induced cycle in $G$, both of which are contradictions.   \end{proof}

For a graph $H$ with $W \subseteq V(H)$, let $\kappa_H(W)$ denote the number of connected components in $H[W]$.

\begin{lemma}\label{lem:betti}
Let $G$ be a chordal graph.  Then $(1)$ If $G$ is connected with $vw \in E(G)$, then $G_{v \to w} \in [G]_{\beta}$; and (2) If $G$ is disconnected with $v,w \in V(G)$ in separate components, then $G_{v \to w} \in [G]_{\beta}$.
\end{lemma}

\begin{proof}
This is a straightforward application of Theorem \ref{thm:froberg} after we make the following calculations.  For each part, let $G' = G_{v \to w}$ and $W \subseteq V(G)$.    

In $(1)$, if $v,w \notin W$, then $\kappa_G(W) = \kappa_{G'}(W)$ since $G \setminus \{v,w\} = G' \setminus \{v,w\}$ and if $v,w \in W$, then $\kappa_G(W) = \kappa_{G'}(W)$ because the component in $G[W]$ containing $v$ and $w$ spans the same set of vertices as that of $G'[W]$.  For the remaining subsets of $V(G)$, we prove that $\kappa_G(W \cup v) + \kappa_G(W \cup w) = \kappa_{G'}(W \cup v) + \kappa_{G'}(W \cup w)$ for every $W \subseteq V(G) \setminus \{v,w\}$.  Let $m_\circ(W)$, $m_w(W)$, and $m_v(W)$ denote the number of connected components of $G[W]$ that do not contain any elements of $N(v) \cup N(w)$, $N(v) \setminus N(w)$, and $N(w) \setminus N(v)$, respectively.  It is straightforward to check that $\kappa_G(W \cup v) = 1 + m_\circ(W) + m_w(W)$,  $\kappa_G(W \cup w) =  1 + m_\circ(W) + m_v(W)$, $\kappa_{G'}(W \cup v) = 1 + m_\circ(W) + m_v(W) + m_w(W)$, and $\kappa_{G'}(W \cup w) = 1 + m_\circ(W)$.

In $(2)$, we record the difference between $\kappa_G(W)$ and $\kappa_{G'}(W)$.  If $v,w \notin W$, then $\kappa_G(W) = \kappa_{G'}(W)$ since $G \setminus \{v,w\} = G' \setminus \{v,w\}$.  If $v,w \in W$, then $\kappa_G(W) = \kappa_{G'}(W)$ because every vertex in the component of $v$ in $G[W]$ gets moved to the component of $w$ in $G'[W]$.  If $v \in W$ and $w \notin W$, then $\kappa_G(W) = \kappa_{G'}(W) -1$.  If $w \in W$ and $v \notin W$, then $\kappa_G(W) = \kappa_{G'}(W) +1$.     
\end{proof}

We are now ready to prove Theorem \ref{thm:threshold}.

\begin{proof}[Proof of Theorem \ref{thm:threshold}]
We induct on $|V(G)|$.  Let $G$ be a chordal graph on $n$ vertices and fix a vertex $v \in V(G)$.  We will apply the operations $-_{v \to w}$ or $-_{w \to v}$ to $G$ to a get a graph where $v$ is either dominating or isolated.

If $G$ is connected and $v$ is not dominating, then for any vertex $u \in G$ with $d(u,v) = 2$, let $w \in N(v) \cap N(u)$ and 
replace $G$ with $G_{w \to v}$.  Repeat this until $v$ is a dominating vertex, i.e. there are no more 
elements $u$ with $d(v,u) = 2$.  The process terminates since $G$ is finite 
and connected.  By Lemma \ref{lem:chordal}, the graph $G$ is chordal at every step and 
by Lemma \ref{lem:betti}, its Betti diagram stays fixed.  Since $v$ is dominating and $G \setminus v$ is chordal (being an induced subgraph of a chordal graph), 
$\beta(\mathbf{k}[G]) = \beta(\mathbf{k}[G \setminus v])$.  So, by induction, there is a unique (up to isomorphism) threshold graph $T$ such that $\beta(\mathbf{k}[T^*]) = \beta(\mathbf{k}[T]) = \beta(\mathbf{k}[G \setminus v]) = \beta(\mathbf{k}[G])$.

If $G$ is disconnected, let $w \in V(G)$ be in a separate component in $G$ from $v$.  By Lemma \ref{lem:chordal} and Lemma \ref{lem:betti}, $G_{v \to w}$ is chordal and $\beta(\mathbf{k}[G]) = \beta(\mathbf{k}[G_{v \to w}])$ and   
by induction, there exists a unique (up to isomorphism) threshold graph 
$T \in [G_{v \to w} \setminus v]_{\beta}$.  Thus, 
$T_* = T \cup \{\alpha\} \in [G]_{\beta}$ and $\beta(\mathbf{k}[T_*]) = \beta(\mathbf{k}[G])$.  
\end{proof}

\begin{remark}
The algorithm presented in the proof of Theorem \ref{thm:threshold} is fast.  A crude analysis of the complexity is as follows:  For each vertex of $G$, we decompose $G$ into its connected components which takes $O(|V(G)|+|E(G)|)$ and then we repeatedly apply the operations $-_{v \to w}$ or $-_{w \to v}$ which, by amortized analysis, takes only $O(|E(G)|)$ since each edge is moved at most once.  Thus, the total complexity is $O(|V(G)|(|V(G)|+|E(G)|)) \approx O(|V(G)|^3)$.  The authors suspect that a more thorough analysis would yield a complexity of $O(|V(G)|^2)$ which is the best one could hope for with this problem.
\end{remark}

As simple corollaries of Theorem \ref{thm:threshold}, we recover two special classes of graphs that are invariant 
under $\beta$.  

\begin{corollary}\label{cor:tree}
If $G$ is a tree on $n+1$ vertices, then $\beta_{i,i+1}(\mathbf{k}[G]) = i{n \choose i+1}$.
\end{corollary}

\begin{proof}
Since $G$ has exactly $n$ edges and $-_{v \to w}$ preserves the number of edges in $G$, the procedure outlined in the proof of Theorem \ref{thm:threshold} yields a threshold representative, $T$, of $G$ that is a star on $n+1$ vertices, i.e. a single dominating vertex $v$ and no other edges. Therefore, $T \setminus v$ consists of $n$ isolated points and, by Proposition \ref{prop:threshold} and Example \ref{ex:iso}, $\beta_{i,i+1}(\mathbf{k}[G]) = \beta_{i,i+1}(\mathbf{k}[T]) = \beta_{i,i+1}(\mathbf{k}[T \setminus v]) = (i){n \choose i+1}$.
\end{proof}

The graph from a triangulation of a polygon is called \emph{maximally outerplanar}.

\begin{corollary}
If $G$ is a maximal outerplanar graph on $n+1$ vertices, then $\beta_{i,i+1}(\mathbf{k}[G]) = i {n-1 \choose i+1}.$
\end{corollary}

\begin{proof}
By Theorem $\ref{thm:threshold}$, the threshold representative $T$ of $G$ consists of a dominating vertex $v$ and a path on $V(T) \setminus v$.  In particular, $T \setminus v$ is a tree on $n$ vertices.  The result now follows from Proposition \ref{prop:threshold} and Corollary \ref{cor:tree}.
\end{proof}

\subsection{Betti diagrams of algebras and modules}

Here we present the main results of the paper -- that every Betti diagram from a $2$-linear ideal in $S$ arises from a Stanley-Reisner ring of a threshold graph on $n+1$ vertices and that every Betti diagram from an $S$-module with a $2$-linear resolution arises from a direct sum of Stanley-Reisner rings constructed from threshold graphs on $n+1$ vertices.  

To begin, we establish bijections between the set of threshold graphs on $n+1$ vertices, the set of Betti diagrams from $2$-linear ideals in $S$, and the set of anti-lecture hall compositions of length $n$ bounded above by $1$.  An integer sequence $\lambda = (\lambda_1,\lambda_2,\dots,\lambda_n)$ of the form $$t \geq \frac{\lambda_1}{1} \geq \frac{\lambda_2}{2} \geq \dots \geq \frac{\lambda_n}{n} \geq 0$$ is called \emph{anti-lecture hall composition of length $n$} bounded above by $t$.  These sequences were introduced in \cite{CS} and are a well-studied variation of the \emph{lecture hall partitions} in \cite{BE1,BE2}.  For our purposes, we only need the following: 

\begin{theorem}[Corteel-Lee-Savage, \cite{CLS}]\label{thm:ALHP}
There are $(t+1)^n$ anti-lecture hall compositions of length $n$ bounded above by $t$.
\end{theorem}

We remark that $\mathbf{k}[G] = R$ if $G$ the complete graph on $n+1$ vertices, so we shall ignore that graph for the rest of the paper.
 
 \begin{proposition}\label{prop:bijection}
 The set of noncomplete threshold graphs on $n+1$ vertices, the set of Betti diagrams of quotients of $S$ by $2$-linear ideals, and the set of anti-lecture hall compositions of length $n$ with $\lambda_1 = 1$ are in bijective correspondence.
 \end{proposition}
 
\begin{proof}   By Lemma \ref{lem:thresholdCount} and Corollary \ref{cor:threshold}, there are $2^{n}-1$ noncomplete 
threshold graphs on $n+1$ vertices, each of which corresponds to a distinct Betti diagram.  It suffices to show that the Betti diagrams of quotients of $S$ by $2$-linear ideals inject into the anti-lecture hall compositions of length $n$ with $\lambda_1 = 1$, since by Theorem \ref{thm:ALHP}, there are exactly $2^n-1$ of them.  

Let $I$ be a 2-linear ideal in $S$ and let $\Psi$ be the unimodular matrix with $ij$--entry equal to ${ i-1 \choose j-1}$.  Then there exists $\lambda = [\lambda_1, \dots, \lambda_{n}] \in \mathbb{Z}^{n}$ such that $\omega(S/I) = \lambda \Psi$.  By Theorem \ref{thm:BS}, $\lambda \Psi \Omega^{-1}
  = [c_1,...,c_{n}] \in \mathbb{Q}_{\geq 0}^{n}$ such that $\sum\limits_{i=1}^{n} c_i = 1$.  
 We leave it to the reader to verify that $\Psi \cdot \Omega^{-1}$ has $ij$--entry $1/i$ if $i = j$, $-1/i$ if $i = j + 1$, and $0$ otherwise.  
   Thus, $\displaystyle c_i = \frac{\lambda_i}{i} - \frac{\lambda_{i+1}}{i+1}$ for all $i \in [n-1]$ 
   and $\displaystyle c_{n} = \frac{\lambda_{n}}{n}$.  In particular, we get that $$\displaystyle 1 = 
   \sum\limits_{i=1}^{n} c_i = \frac{\lambda_1}{1} \geq \frac{\lambda_2}{2} \geq ... 
   \geq \frac{\lambda_{n}}{n} = c_{n} \geq 0$$ and hence, $\lambda$ is an anti-lecture hall composition with $\lambda_1 = 1$. \end{proof}


The first part of our main theorem is a simple corollary of Proposition \ref{prop:bijection}.  In particular, it asserts that the injection in Proposition \ref{prop:threshold} is in fact a bijection.
 
\begin{theorem}[Main Theorem, Part 1]\label{thm:alg}
For every $2$-linear ideal $I$ in $S$, there is a unique threshold graph $T$ on $n+1$ vertices with $\beta(S/I) = \beta(\mathbf{k}[T])$.
\end{theorem}


\begin{remark} For a given $2$-linear ideal $I$ in $S$, it is easy to construct the graph $T$ realizing its Betti diagram. \end{remark}

\begin{example}\label{ex:alg}
To illustrate Theorem \ref{thm:alg} at work, consider the ideal $$I = \langle x_1^2, x_1x_2, x_1x_3, x_1x_4, x_2^2, x_1x_5+x_2x_4, x_4^2 \rangle \subseteq S = \mathbf{k}[x_1,\dots,x_5].$$  Then $$\beta(S/I) = \begin{bmatrix} 1 & \cdot &  \cdot &  \cdot & \cdot & \cdot \\  \cdot & 7 & 11 & 6 & 1 & 0 \end{bmatrix}.$$
In order to find a threshold graph $T$ on six vertices whose Betti diagram is $\beta(S/I)$, we sequentially apply the inverses of the affine transformations in Proposition \ref{prop:threshold} depending on whether or not the sequences end in $0$.  (We leave it to the reader to verify that the inverse of $\Lambda$ in Proposition \ref{prop:threshold} is the $n \times (n-1)$-matrix whose $(i,j)$ position is $(-1)^{i+j}$ if $i \leq j$ and $0$ otherwise.)
$$[7,11,6,1,0] \overset{-^*}{\longrightarrow} [7,11,6,1]  \overset{-_*}{\longrightarrow} [3,2,0] \overset{-^*}{\longrightarrow}  [3,2]  \overset{-_*}{\longrightarrow} [1] \overset{-_*}{\longrightarrow} [0] $$
From this, we see that $\beta(S/I) = \beta(\mathbf{k}[T])$ where $T$ is the threshold graph with sequence $((((-_*)_*)^*)_*)^*$ drawn in Figure \ref{fig:ex}.
\end{example} 

\begin{figure} [ht]
\begin{center}\scalebox{1.55}{\includegraphics{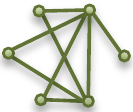}} \end{center}
\caption{The threshold graph $T$ on six vertices with $\omega(\mathbf{k}[T]) = [7,11,6,1,0]$.}\label{fig:ex}
 \end{figure} 

For the rest of the paper, we take a more geometric approach.  Specifically, we make use of the fact (Remark \ref{rem:lattice}) that the reduced Betti vectors of these diagrams are lattice points in the $(n-1)$-dimensional lattice simplex $P_n$ spanned by the row vectors of $\Omega$.  Illustrations of $P_1$ through $P_4$, labeled by reduced Betti vectors, Boij-S\"oderberg coefficients, truncated coordinates (see Section 5), and corresponding chordal graphs are shown in Figures \ref{fig:P_1} through $\ref{fig:P_4}$ with the threshold graphs colored dark green.  Notice that each $P_n$ contains two copies of $P_{n-1}$, colored blue and red, corresponding to Parts (1) and (2) of Proposition \ref{prop:threshold}, respectively.   

\begin{figure}[ht] \begin{minipage}[b]{0.45\linewidth} \begin{center} \includegraphics[width=0.1\textwidth]{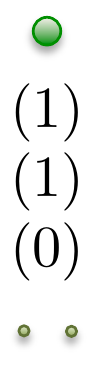}  \end{center}  \caption{The lattice polytope $P_1$.}\label{fig:P_1} \end{minipage}  \hspace{0.1cm}
\begin{minipage}[b]{0.5\linewidth} \begin{center}   \includegraphics[width=5cm]{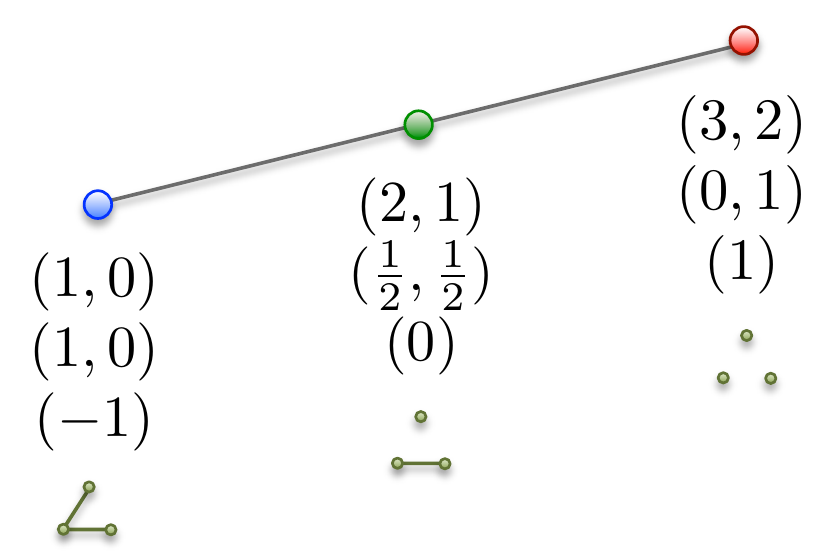}  \end{center} \caption{The lattice polytope $P_2$.}\label{fig:P_2} \end{minipage} \end{figure}

 \begin{figure}[ht] \begin{center} \includegraphics[width=8.75cm]{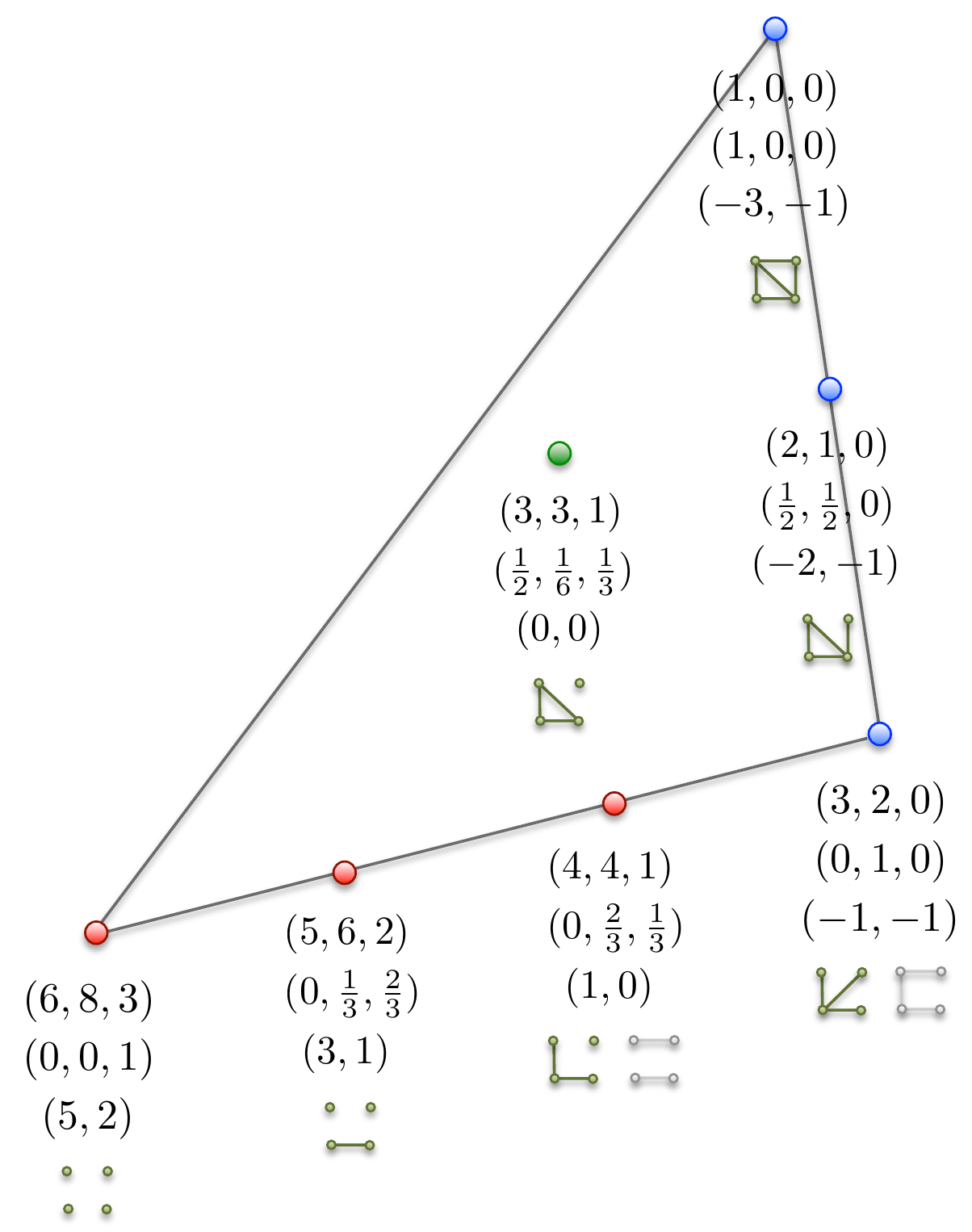} \end{center}  \caption{The lattice polytope $P_3$.}\label{fig:P_3} \end{figure}

 \begin{figure}[ht] \begin{center}  \includegraphics[width=12cm]{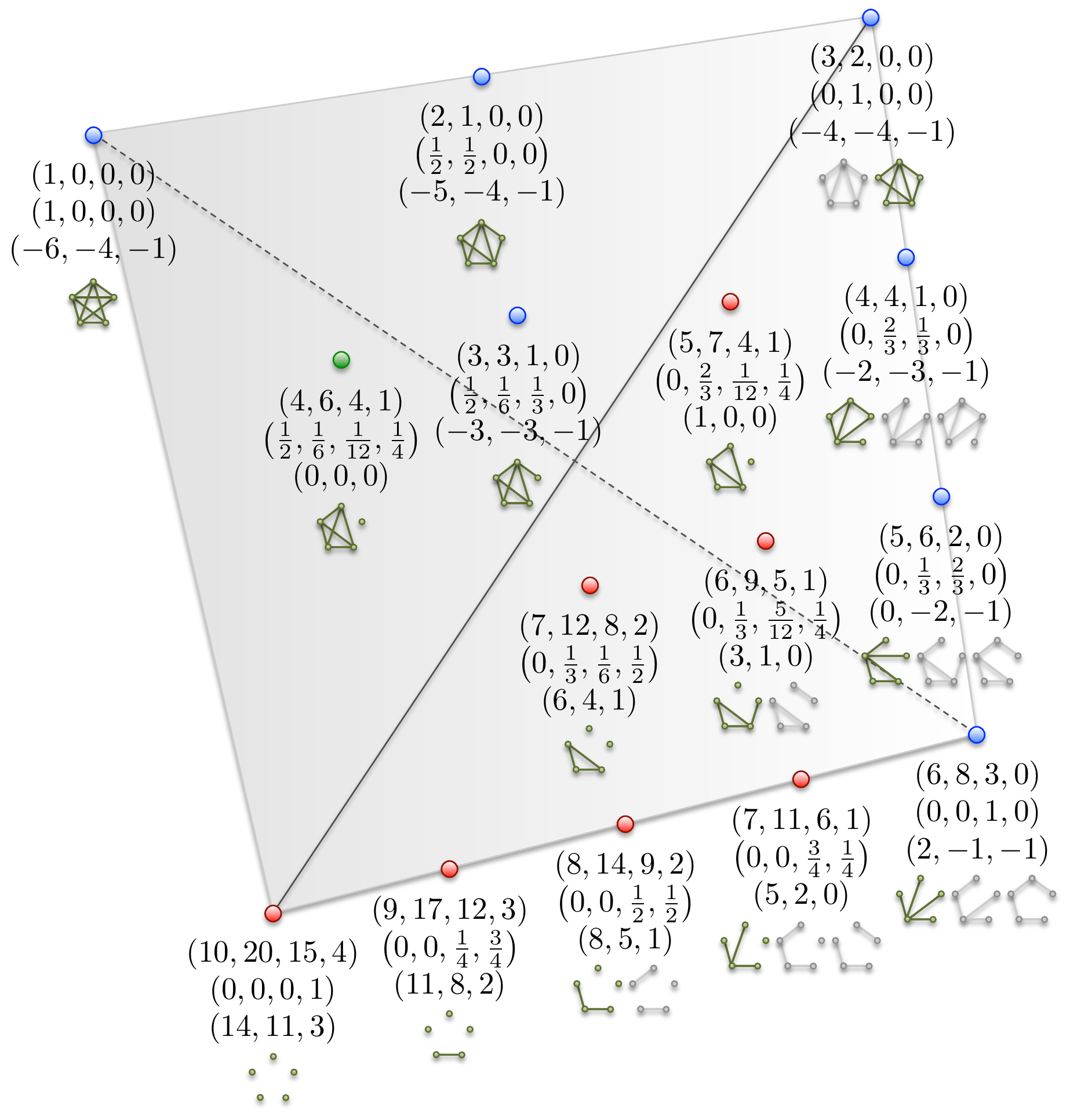} \end{center} \caption{The lattice polytope $P_4$.}\label{fig:P_4} \end{figure}

We continue with some standard definitions from discrete geometry.  The integer points $\mathbb{Z}^d \subseteq \mathbb{R}^d$ form a \emph{lattice}. The integer points of a polytope are its \emph{lattice points} and a polytope is called a \emph{lattice polytope} if all its vertices are lattice points.  For a polytope $P$ with vertices $\{v_1,\dots,v_s\}$ and $t \in \mathbb{N}$, let $t P$ denote  the $t^{\text{th}}$ \emph{dilation} of $P$, i.e. the polytope attained by taking the convex hull of the points $\{t \cdot v_1, \dots, t \cdot v_s\}$, let $S_P \subseteq \mathbb{Z}^{d+1}$ denote the semigroup generated by $\{ [ 1,p_1,\dots, p_d  ] \ : \ (p_1,\dots, p_d) \in P \cap \mathbb{Z}^d \}$, and let $\text{gp}(S_P)$ be the smallest group containing $S_P$, i.e. the group of differences in $S_P$.  We say $P$ is \emph{normal} if $x \in \text{gp}(S_P)$ such that $s \cdot x \in S_P$ for some $s \in \mathbb{N}$ implies that $x \in S_P$.  We refer to \cite{barvinok2002,BGT} for questions on lattice polytopes.

\begin{proposition}\label{prop:normal}
The lattice simplex $P_n$ is normal for each $n \in \mathbb{N}$.
\end{proposition}

\begin{proof}
It is straightforward to check that the anti-lecture hall compositions of length $n$ bounded above by $1$ are the lattice points of the $n$-dimensional lattice simplex spanned by $(0, \dots, 0)$ and the compositions $\lambda^l = (1,2,\dots,l,0,\dots,0)$ for $l \in [n]$.  Let $Q_n$ be the facet spanned by the $\lambda^l$'s.  Since normality is preserved under unimodular transformations, we prove that $Q_n$ is normal and apply $\Psi$ from the proof of Proposition \ref{prop:bijection}.

To begin, we must truncate the coordinates of $Q_n$ since it is an $(n-1)$-dimensional simplex.  Removing the first coordinate yields the simplex with vertices $(0,\dots,0)$ and $(2,3,\dots,l,0,\dots,0)$ for $l \in [n]$.  Then $S_{Q_n}$ is the set of all anti-lecture hall compositions and $\text{gp}(S_{Q_n}) = \mathbb{Z}^n$.  From here it is clear that if $\lambda \in \mathbb{Z}^n$ and $s \cdot \lambda \in S_{Q_n}$ for some $s \in \mathbb{N}$, then $\lambda \in S_{Q_n}$.  Hence, $Q_n$ is normal.
\end{proof}

A convenient consequence of normality is that every lattice point in the $t^{\text{th}}$ dilation of a normal polytope $P$ can be written as a sum of $t$, not necessarily distinct, lattice points in $P$.  With that, we can prove the second part of our main theorem.  

\begin{theorem}[Main Theorem, Part 2] \label{thm:module}
For every finitely-generated, graded $S$-module, $M$, with a $2$-linear minimal free resolution and $\beta_{0,0}(M) = m$, there exists a collection of $m$ threshold graphs $\{T_1,\dots,T_m\}$, not necessarily distinct, such that $\beta(M) =  \beta(\mathbf{k}[T_1] \oplus \cdots \oplus \mathbf{k}[T_m]).$ 
\end{theorem}

\begin{proof}
By Theorem \ref{thm:BS}, $\omega(M)$ is a lattice point in $m P_n$ and is a sum of $m$ lattice points $p_1$, \dots , $p_m$ in $P_n$, by Proposition \ref{prop:normal}.  Applying Theorem \ref{thm:alg} yields a threshold graph $T_i$ such that $p_i = \omega(\mathbf{k}[T_i])$ for each $i \in [m]$, and thus, $$\beta(M) = \beta(T_1) + \cdots +\beta(T_m) = \beta(\mathbf{k}[T_1] \oplus \cdots \oplus \mathbf{k}[T_m]).$$
\end{proof}

\begin{remark}  
The decomposition in Theorem \ref{thm:module} is often not unique.  So in the more general setting of modules, we do not know how to construct the family of trees representing a given Betti diagram as we do in the special case of algebras, see Theorem \ref{thm:alg} and Example \ref{ex:alg}.
\end{remark}

\section{The geometry of $P_n$ and $Q_n$}

In the previous section, we used the geometry of the lattice simplex $P_n$ of reduced Betti vectors of $2$-linear ideals in $S$ (or equivalently, the lattice simplex $Q_n$ of nonzero anti-lecture hall compositions of length $n$) to prove algebraic statements about Betti diagrams of algebras and modules with $2$-linear resolutions, but these polytopes have many other beautiful geometric properties which make them interesting on their own.  In this section, we take the opportunity to showcase a few of these properties.  Specifically, we remark that $P_n$ has a simple Ehrhart polynomial, due to Corteel, Lee, and Savage, and we prove that $P_n$ is reflexive.




Given a $d$-dimensional polytope $P$, let $Ehr_{P}(t)$ denote the number of lattice points in $t P$.  It is well-known that $Ehr_{P}(t)$ is a degree $d$ polynomial in $t$, called the \emph{Ehrhart polynomial} of $P$, with constant term $1$ and leading coefficient equal to the volume of $P$, and that Ehrhart polynomials are preserved under unimodular transformations.   For an introduction to Ehrhart theory, see \cite{BR}. 

\begin{theorem} For every $n,t \in \mathbb{N}$,
$Ehr_{P_n}(t) = Ehr_{Q_n}(t) = (t+1)^{n} - t^{n}$. 
\end{theorem}

\begin{proof}
Since the matrix $\Psi^{-1}$ in the proof of Proposition \ref{prop:bijection} is unimodular, we know that $Ehr_{P_n}(t) = Ehr_{Q_n}(t)$.  So, let $A_n(t)$ denote the number of anti-lecture hall compositions of length $n$ with $\lambda_1 \leq t$.  Theorem \ref{thm:ALHP} gives us $A_n(t) = (t+1)^n$.  Since every point in the $tQ_n$ satisfies, $\lambda_1 = t$, it follows immediately that  $$Ehr_{P_n}(t) = Ehr_{Q_n}(t) = A_n(t) - A_{n}(t-1) = (t+1)^n - t^n.$$ 
\end{proof}

Next, we prove that $P_n$ is reflexive.  For this, we need the concept of a dual (or polar) of a polytope, but restrict to the case of simplices, since those are the only polytopes we consider.  

\begin{definition}
Let the vertices of a $d$-simplex $P$ be recorded as the rows
of the $d \times (d-1)$ matrix $M$ and let $M^\ast$ be the
$(d-1) \times d$ matrix such that $MM^\ast$ has value -1
everywhere outside the diagonal. The $d$-simplex whose
vertices are the columns of $M^\ast,$ is the \emph{dual} $P^\ast$ of $P$.
\end{definition}

If $P$ is a lattice polytope containing 0 as an interior point such that $P^\ast$ is also lattice polytope, then $P$ and $P^{\ast}$ are called \emph{reflexive}.  These polytopes have several interesting properties and characterizations, for instance, a lattice polytope $P$ is reflexive if and only if its only interior lattice point is $0$ and if $u$ and $v$Ê are two lattice points on the boundary of $P$, then either $u$ and $v$ are
on the same facet, or $u+v$ is in $P$.  This is an important concept with interesting
connection to geometry and theoretical physics. For an exposition
suitable for researchers with a background in discrete
mathematics, we refer to Batyrev and Nill~\cite{batyrevNill2008}.

Because $P_n$ is an $(n-1)$-dimensional simplex with coordinates in $\mathbb{Z}^n$, for each lattice point $p \in P_n$, we define $$p_t = [p_1,\dots,p_{n-1}] := [p - \eta_n]_{2 \leq i \leq n} = \left[p_2 - \binom{n}{2},\dots,p_n-\binom{n}{n}\right]$$ to be the \emph{truncated coordinates} of $p$ in $P_n$. 

\begin{theorem}\label{thm:reflexive} The simplex $P_n$ realized in the truncated coordinates is a reflexive 
lattice polytope.
\end{theorem}
\begin{proof}
We begin by removing the left-most column of $\Omega$ to get the $n \times (n-1)$ matrix $\Omega'_{n}.$ Then the truncated coordinates of $P_n$ are the rows of $\Omega_{n} = \Omega'_{n} - \eta_{n} \mathbf{1}_{n}.$ More explicitly, the $ij$-entry of $\Omega'_{n}$ is $(j+1){i+1 \choose j+2}$ and the $j$ entry 
of $\eta_{n}$ is ${n \choose j+1}.$

The dual of $P_n$, in truncated coordinates, is the simplex whose vertices are 
the columns of the $(n-1)\times (n)$ matrix $\Xi_{n}$ satisfying that all values of
 $\Omega_{n}\Xi_{n}$ outside the diagonal is $-1.$ If all entries of $\Xi_n$ are integers, 
 then the dual of $P_n$ is a lattice polytope and hence, $P_n$ is reflexive.  To show this, we 
 construct $\Xi_{n}$ explicitly with three $(n-1) \times n$ matrices, $\Xi_n'$, $\Xi_n''$, and $\Xi_n'''$.  
The $ij$-entries of $\Xi_n'$ are $-(i+2)(-1)^{i+j}{i \choose j-1}$ and the matrices $
\Xi_n''$ and $\Xi_n'''$ are all zero, with the exceptions that the first column of $\Xi_n''$ is 
$-2(-1)^i,$ and the bottom right-most entry
of $\Xi_n'''$ is $1-n.$ We consider $\Xi_n = \Xi_n' + \Xi_n'' + \Xi_n'''$. 

To calculate the product  $\Omega_n\Xi_n$, we separate both $\Omega_n$ and 
$\Xi_n$ into the sums above and then multiply them. The matrix multiplications are 
straightforward applications of elementary combinatorics, so we only record the 
results:
\begin{itemize}
\item[1)] The matrix $\Omega_n'\Xi_n'$ is the sum of two matrices: The only non-zero 
elements of the first one are the diagonal $ii$-entries $i(i+1)$ and the only non-zero 
elements of the second one are the first column $i1$-entries $-i(i+1).$
\item[2)] The matrix $\eta_{n} \mathbf{1}_{n}\Xi_n'$ is an all ones matrix, except 
for that the first column is constant $-2n+1$ and the last column is $n+1.$
\item[3)] The matrix $\Omega_n'\Xi_n''$ is an all zero matrix, except for that the first 
column $i1$-entry is $i(i+1)-2.$
\item[4)] The matrix $\eta_{n} \mathbf{1}_{n}\Xi_n''$ is an all zero matrix, except 
for that the first column is constant $2n-2.$
\item[5)] The matrix $\Omega_n'\Xi_n'''$ is an all zero matrix, except for that the 
rightmost bottom corner is $-n^2.$
\item[6)] The matrix $\eta_{n} \mathbf{1}_{n}\Xi_n'''$ is an all zero matrix, except 
for that the rightmost column is constant $-n.$
\end{itemize}
Summing up, we conclude that
the $ij$-entry of $\Omega_n\Xi_n=( \Omega'_n - \eta_{n} \mathbf{1}_{n})( \Xi_n' + 
\Xi_n'' + \Xi_n''')$ is
\[ -1 \textrm{ if }i\neq j,  \quad\quad i^2+i-1 \textrm{ if }i=j<n, \quad \quad n \textrm
{ if }i=j=n. \]
\end{proof}


\medskip

\begin{ack}
The authors thank Mats Boij, for several insightful conversations regarding this work, Benjamin Braun, for pointing out the connection to anti-lecture hall compositions, and the anonymous referee, for the helpful suggestions to improve this paper.  Alexander Engstr\"om thanks the Miller Institute for Basic Research at UC Berkeley for funding.  Matthew Stamps thanks the Mathematical Sciences Research Institute for support to attend the 2011 Summer Graduate Workshop on Commutative Algebra.
\end{ack}


\begin{thebibliography}{10}

\bibitem{barvinok2002}
Alexander Barvinok.
\emph{A course in convexity.}
Graduate Studies in Mathematics, 54. American Mathematical Society, Providence, RI, 2002. 366 pp.

\bibitem{batyrevNill2008}
Victor Batyrev and Benjamin Nill.
Combinatorial aspects of mirror symmetry. \emph{Integer points in polyhedra---
geometry, number theory, representation theory, algebra, optimization, statistics}, 35--66,  
Contemp. Math., 452, Amer. Math. Soc., Providence, RI, 2008.

\bibitem{BR}
Matthias Beck and Sinai Robins.
\emph{Computing the Continuous Discretely: Integer-Point Enumeration in Polyhedra.}  Undergraduate Texts in Mathematics.  Springer, New York, 2007. 226 pp.

\bibitem{BS2} 
Mats Boij and Jonas S\"oderberg.
Betti numbers of graded modules and the Multiplicity Conjecture in the non-Cohen-Macaulay case.
\emph{Algebra Number Theory} {\bf 6} (2012), no. 3, 437--454.

\bibitem{BS1} 
Mats Boij and Jonas S\"oderberg.
Graded Betti numbers of Cohen-Macaulay modules and the multiplicity conjecture.
\emph{J. Lond. Math. Soc. (2)} {\bf 78} (2008), no. 1, 78--101. 

\bibitem{BE1}
Mireille Bousquet-M\'elou and Kimmo Eriksson.
Lecture hall partitions.
\emph{Ramanujan J.} {\bf 1} (1997), no. 1, 101--111.

\bibitem{BE2}
Mireille Bousquet-M\'elou and Kimmo Eriksson.
Lecture hall partitions II.
\emph{Ramanujan J.} {\bf 1} (1997), no. 2, 165--185.


\bibitem{BGT}
Winfried Bruns, Joseph Gubeladze, and Ng\^{o} Vi\^{e}t Trung.
Normal polytopes, triangulations, and Koszul algebras.
\emph{J. Reine Angew. Math.} {\bf 485} (1997), 123--160.

\bibitem{CS}
Sylvie Corteel and Carla D. Savage. 
Anti-lecture hall compositions.
\emph{Discrete Math.} {\bf 263} (2003),  no. 1-3, 275--280.

\bibitem{CLS}
Sylvie Corteel, Sunyoung Lee, and Carla D. Savage. 
Enumeration of sequences constrained by the ratio of consecutive parts.
\emph{S\'em. Lothar. Combin.} {\bf 54A} (2005/07), Art. B54Aa, 12 pp.

\bibitem{diestel2012}
Reinhard Diestel. \emph{Graph Theory.} Springer-Verlag, Heidelberg
Graduate Texts in Mathematics, Volume 173. 4th edition. Corrected reprint 2012. 451 pp.

\bibitem{DE}
Anton Dochtermann and Alexander Engstr\"om.
Algebraic properties of edge ideals via combinatorial topology.
\emph{Electron. J. Combin.} {\bf 16} (2009), no. 2, Research Paper 2, 24 pp. 

\bibitem{ES} David Eisenbud and Frank-Olaf Schreyer.
Betti numbers of graded modules and cohomology of vector bundles. 
\emph{J. Amer. Math. Soc.} {\bf 22} (2009), no. 3, 859--888. 

\bibitem{floystad} Gunnar Fl\o ystad.
Boij-S\"oderberg theory: introduction and survey.
\emph{Progress in commutative algebra 1}, 1--54, de Gruyter, Berlin, 2012.

\bibitem{froberg90} Ralf Fr\"oberg.
On Stanley-Reisner rings. \emph{Topics in algebra}, Part 2 (Warsaw, 1988), 57--70, 
Banach Center Publ., 26, Part 2, PWN, Warsaw, 1990.

\bibitem{GY} Afshin Goodarzi and Siamak Yassemi.
Shellable quasi-forests and their $h$-triangles.  
\emph{Manuscripta Math.}, 137 (2012), 475--481.

\bibitem{HSV} J\"urgen Herzog, Leila Sharifan, and Matteo Varbaro.
Graded Betti numbers of componentwise linear ideals.
\emph{Proc. Amer. Math. Soc.}, to appear,
{\tt arXiv:1111.0442}, 21 pp.


\bibitem{klivans07} Caroline J. Klivans.
Threshold graphs, shifted complexes, and graphical complexes.
\emph{Discrete Math.}  {\bf 307} (2007), 2591--2597.

\bibitem{mahadevPeled1995}
Nadimpalli Mahadev and Uri Peled. 
\emph{Threshold graphs and related topics}. 
Annals of Discrete Math, 56. North-Holland, Amsterdam, 1995. 543 pp.

\bibitem{NS} Uwe Nagel and Stephen Sturgeon.
Combinatorial Interpretations of some Boij-S\"oderberg Decompositions.
\emph{J. Algebra} {\bf 381}, 54--72.

\bibitem{woodroofe11} Russ Woodroofe.
Erd\H{o}s-Ko-Rado theorems for simplicial complexes.
\emph{J. Combin. Theory, Series A} {\bf 118} (2011), no. 4, 1218--1227.

\end{thebibliography}
\end{document}